\documentclass[reqno]{amsart}

\usepackage[utf8]{inputenc}
\usepackage[T1]{fontenc}
\usepackage[british]{babel,isodate}
\cleanlookdateon
\usepackage{lmodern}
\usepackage{enumerate}
\usepackage{tikz}
\usepackage{tikz-cd}
\usetikzlibrary{arrows}
\usepackage{cancel}
\usetikzlibrary{babel}
\usepackage{lscape}
\usepackage{url}

\usepackage{hyperref}
\hypersetup{linktocpage,colorlinks=true,citecolor=purple,linkcolor=blue,urlcolor=blue,filecolor=black}

\usepackage{float}
\usepackage{graphicx, import}
\usepackage{mathtools}
\usepackage{dsfont}
\usepackage{pinlabel}
\usepackage{yhmath}
\usepackage{enumitem}

\usepackage{xargs}                      % Use more than one optional parameter in a new commands
\usepackage{xcolor} % Coloured text etc.
\definecolor{OliveGreen}{rgb}{0,0.6,0}
\definecolor{Goldenrod}{rgb}{1,0,0}
\definecolor{rosa}{rgb}{0.788, 0, 0.616}

\usepackage[colorinlistoftodos, prependcaption,textsize=tiny]{todonotes}
\newcommandx{\unsure}[2][1=]{\todo[linecolor=red,backgroundcolor=red!25,bordercolor=red,#1]{#2}}
\newcommandx{\change}[2][1=]{\todo[linecolor=blue,backgroundcolor=blue!25,bordercolor=blue,#1]{#2}}
\newcommandx{\info}[2][1=]{\todo[linecolor=OliveGreen,backgroundcolor=OliveGreen!25,bordercolor=OliveGreen,#1]{#2}}
\newcommand{\centre}[1]{\begin{array}{c} #1 \end{array}}

\usepackage{stackengine}

\usepackage{amsmath,amsthm,amssymb, amsfonts, amscd}
\usepackage{rotating,subcaption}
\usepackage{url}
\usepackage{graphicx,bm}
\usepackage{mathtools}
\usepackage[mathscr]{euscript}
\allowdisplaybreaks
\usepackage{lipsum}
\usepackage{bm}
\usepackage[capitalize]{cleveref}

\usepackage{xpatch}
\makeatletter
\AtBeginDocument{\xpatchcmd{\@thm}{\thm@headpunct{.}}{\thm@headpunct{}}{}{}}
\makeatother

\newtheorem{theorem}{Theorem}[section]
\newtheorem{proposition}[theorem]{Proposition}
\newtheorem{lemma}[theorem]{Lemma}
\newtheorem{corollary}[theorem]{Corollary}

\theoremstyle{definition}

\newtheorem{example}[theorem]{Example}

\theoremstyle{remark}

\newtheorem{remark}[theorem]{Remark}

\numberwithin{equation}{section}

\newcommand{\id}{\mathrm{Id}}

\renewcommand{\to}{\longrightarrow}

\usepackage{accents}
\newcommand{\uhat}{\underaccent{\check}}

\makeatletter
\newcommand{\cupr@tip}{\text{\raisebox{-0.1ex}{$\m@th\hat{}$}}}
\newcommand{\cupr}{\mathbin{\cup\cupr@}}

\newcommand{\cupr@}{%
  \mathchoice
  {\mkern-1.35mu\cupr@tip}
  {\mkern-1.35mu\cupr@tip}
  {\mkern-1.55mu\cupr@tip}
  {\mkern-1.875mu\cupr@tip}
}
\makeatother

\makeatletter
\newcommand{\capr@tip}{\text{\raisebox{0.47ex}{$\m@th\uhat{}$}}}
\newcommand{\capr}{\mathbin{\capr@\cap}}

\newcommand{\capr@}{%
  \mathchoice
  {\mkern11.6mu\capr@tip\mkern-11.6mu}
  {\mkern11.4mu\capr@tip\mkern-11.4mu}
  {\mkern11.1mu\capr@tip\mkern-11.1mu}
  {\mkern10.2mu\capr@tip\mkern-10.2mu}
}
\makeatother

\makeatletter
\newcommand{\capl@tip}{\text{\raisebox{0.47ex}{$\m@th\uhat{}$}}}
\newcommand{\capl}{\mathbin{\capl@\cap}}

\newcommand{\capl@}{%
  \mathchoice
  {\mkern2.1mu\capl@tip\mkern-2.1mu}
  {\mkern2.1mu\capl@tip\mkern-2.1mu}
  {\mkern2.3mu\capl@tip\mkern-2.3mu}
  {\mkern2.1mu\capl@tip\mkern-2.1mu}
}
\makeatother

\makeatletter
\newcommand{\cupl@tip}{\text{\raisebox{-0.1ex}{$\m@th\hat{}$}}}
\newcommand{\cupl}{\mathbin{\cupl@\cup}}

\newcommand{\cupl@}{%
  \mathchoice
  {\mkern1.35mu\cupl@tip\mkern-1.35mu}
  {\mkern1.35mu\cupl@tip\mkern-1.35mu}
  {\mkern1.55mu\cupl@tip\mkern-1.55mu}
  {\mkern1.875mu\cupl@tip\mkern-1.875mu}
}
\makeatother

\DeclareFontFamily{U}{mathx}{}
\DeclareFontShape{U}{mathx}{m}{n}{ <-> mathx10 }{}
\DeclareSymbolFont{mathx}{U}{mathx}{m}{n}
\DeclareFontSubstitution{U}{mathx}{m}{n}
\DeclareMathAccent{\widecheck}{0}{mathx}{"71}

\usetikzlibrary{decorations.markings}
\tikzset{double line with arrow/.style args={#1,#2}{decorate,decoration={markings,%
mark=at position 0 with {\coordinate (ta-base-1) at (0,1pt);
\coordinate (ta-base-2) at (0,-1pt);},
mark=at position 1 with {\draw[#1] (ta-base-1) -- (0,1pt);
\draw[#2] (ta-base-2) -- (0,-1pt);
}}}}

\begin{document}

%%
%% The title of the paper goes here.  Edit to your title.
%%
%The 2-loop polynomial of genus 1 knots coming from Hopf algebras

\title[Knot invariants from XC-structures on  $SW$  are trivial]{Knot invariants from XC-structures on the Sweedler algebra are trivial}

%Triviality of XC-structures on the Sweedler algebra

%\title{Triviality of XC-structures on the Sweedler algebra}
\date{\today}
%%
%% Now edit the following to give your name and address:
%% 

\author{Jorge Becerra}
\address{Université Bourgogne Europe, CNRS, IMB UMR 5584, F-21000 Dijon, France}
\email{\href{mailto:Jorge.Becerra-Garrido@ube.fr}{Jorge.Becerra-Garrido@ube.fr}}
\urladdr{ \href{https://sites.google.com/view/becerra/}{https://sites.google.com/view/becerra/}} % Delete if not wanted.
%\thanks{A mi madre}
%%
%% If there is another author uncomment and edit the following.
%%

%\author{Second Author}
%\address{Department of Mathematics, University of South Carolina,
%Columbia, SC 29208}
%\email{second@math.sc.edu}
%\urladdr{www.math.sc.edu/$\sim$second}

%%
%% If there are three of more authors they are added in the obvious
%% way. 
%%

%%%
%%% The following is for the abstract.  The abstract is optional and
%%% if not used just delete, or comment out, the following.
%%%

\begin{abstract}
An XC-algebra is the minimum algebraic structure needed to define a framed, oriented knot invariant and generalises Lawrence's invariant obtained from ribbon Hopf algebras. In this note, we show that the knot invariant produced by any XC-structure on the Sweedler algebra is completely determined by the framing of the knot. Furthermore, we also exhibit explicit families of XC-structures on the Sweedler algebra that do not have a ribbon Hopf-algebraic origin.
\end{abstract}

\keywords{XC-algebra, Sweedler algebra, knot invariants}
\subjclass{16T05, 16T30, 18M15, 57K10, 57K16}

%%
%%  LaTeX will not make the title for the paper unless told to do so.
%%  This is done by uncommenting the following.
%%

\maketitle

%%
%% LaTeX can automatically make a table of contents.  This is done by
%% uncommenting the following:
%%
\setcounter{tocdepth}{1}
\tableofcontents

%%
%%  To enter text is easy.  Just type it.  A blank line starts a new
%%  paragraph. 
%%

\section{Introduction}

It is folklore in quantum topology that \textit{ribbon Hopf algebras} give rise to knot invariants in two different ways: one using their representation theory, via the celebrated \textit{Reshetikhin-Turaev functor} \cite{RT}, and the other using the algebra itself, via Lawrence's \textit{universal invariant} \cite{lawrence,habiro}.   It is well-known that the latter  dominates the family of Reshetikhin-Turaev invariants,
% for all representations of the ribbon Hopf algebra
 and furthermore it is the invariant underlying the construction of one of the strongest polynomial-time knot polynomial invariants up to date \cite{barnatanveenpolytime, barnatanveengaussians, becerra_gaussians}.

However, a much weaker algebraic structure is in fact needed to produce an isotopy invariant of framed, oriented, long knots following the construction of the universal invariant, namely an \textit{XC-algebra}
\cite{becerra_thesis, becerra_refined,BH_reidemeister, becerra_XC,bosch}.
%An XC-algebra is the minimal algebraic structure needed to produce an invariant of oriented, framed knots à la Lawrence.
We recall the definition here: given an algebra $A$ over some ring $\Bbbk$, an \textit{XC-structure} on $A$ is the choice of two invertible elements $$R \in A\otimes A \qquad , \qquad \kappa \in  A  $$ 
satisfying 
\begin{enumerate}[leftmargin=4\parindent, itemsep=2mm]
\item[(XC0)] $R^{\pm 1}=(\kappa \otimes \kappa) \cdot R^{\pm 1} \cdot (\kappa^{-1} \otimes \kappa^{-1})$,
\item[(XC1f)] $\sum_i \beta_i \kappa \alpha_i = \sum_i \alpha_i \kappa^{-1} \beta_i   $ ,
\item[(XC2c)] $  1\otimes \kappa^{-1} =  \sum_{i,j}  \alpha_i \bar{\alpha}_j \otimes \bar{\beta}_j \kappa^{-1} \beta_i    $,
\item[(XC2d)] $\kappa \otimes 1 =   \sum_{i,j}   \bar{\alpha}_i \kappa  \alpha_j \otimes \beta_j \bar{\beta}_i$,
\item[(XC3)] $R_{12}R_{13}R_{23}=R_{23}R_{13}R_{12}$,
\end{enumerate}
where we have put $R = \sum_i \alpha_i \otimes \beta_i$ and $R^{-1} = \sum_i \bar{\alpha}_i \otimes \bar{\beta}_i$. The element $\nu := \sum_i \beta_i \kappa \alpha_i$ from (XC1f) is called the \textit{inverse of the classical ribbon element}\index{inverse of the classical ribbon element}. The triple $(A, R, \kappa)$ consisting of  a $\Bbbk$-algebra $A$ and an XC-structure on it is called an \textit{XC-algebra}. Ribbon Hopf algebras and endomorphism algebras of finite-dimensional representations of these are natural examples of XC-algebras \cite{becerra_refined}.

Let $K$ be an (oriented, framed, long) knot in a cube. We will be interested in \textit{rotational diagrams} of $K$, that is, diagrams with the property that every point has a neighbourhood where the diagram looks exactly like one of the following building blocks:
\begin{equation}\label{eq:crossings_and_spinners}
\centre{
\labellist \small \hair 2pt
\pinlabel{$I$}  at 20 -78
\pinlabel{$X$}  at 350 -78
 \pinlabel{$ X^- $}  at 800 -78
 \pinlabel{$ C$}  at 1140 -78
  \pinlabel{$ C^- $}  at 1560 -78
\endlabellist
\centering
\includegraphics[width=0.6\textwidth]{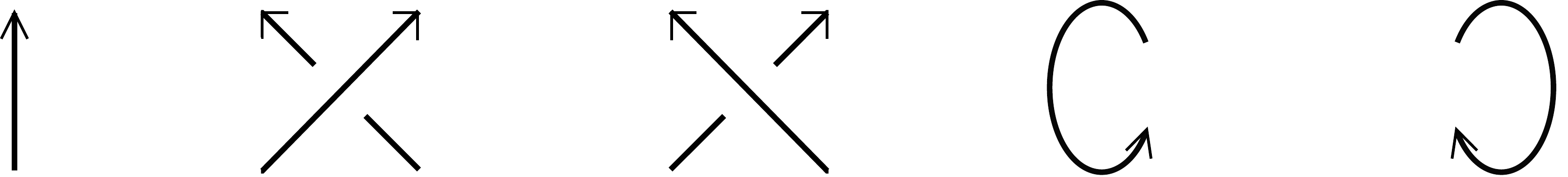}}
\end{equation}
\vspace*{0.5cm}

\noindent By \cite[Lemma 3.2]{becerra_refined}, any knot has a rotational diagram. Furthermore, we can always remain in the realm of rotational knot diagrams by considering sets of \textit{rotational Reidemeister moves}, see \cite{BH_reidemeister} for an account.

Let us recall now the  construction of the universal invariant subject to a given XC-algebra $(A, R, \kappa)$. Given  a knot $K$, consider $D$ a rotational diagram of it, and decorate the building blocks $I, X^{\pm}, C^{\pm}$ from above with beads representing the elements $1, R^{\pm 1}, \kappa^{\mp 1}$ as follows: 
\begin{equation}\label{eq:beads}
\centre{
\labellist \small \hair 2pt
\pinlabel{$ \color{violet} \bullet$} at 16 30
\pinlabel{$ \color{violet} 1$} at -26 30
\pinlabel{$ \color{violet} \bullet$} at 294 30
\pinlabel{$ \color{violet} \bullet$} at 406 30
\pinlabel{$ \color{violet} \alpha_i$} [r] at 279 30
\pinlabel{$ \color{violet} \beta_i$} [l] at 421 30
\pinlabel{$ \color{violet} \bullet$} at 711 30
\pinlabel{$ \color{violet} \bullet$} at 825 30
\pinlabel{$ \color{violet} \bar{\beta}_i$} [r] at 696 30
\pinlabel{$ \color{violet} \bar{\alpha}_i$} [l] at 850 30
\pinlabel{$ \color{violet} \bullet$} at 1083 92
\pinlabel{$ \color{violet} \kappa^{-1}$} [r] at 1090 120
\pinlabel{$ \color{violet} \bullet$} at 1599 92
\pinlabel{$ \color{violet} \kappa$} [l] at 1614 104
\endlabellist
\centering
\includegraphics[width=0.6\textwidth]{building_blockss}}
\end{equation}
\vspace*{-10pt}

\noindent Then let $\mathfrak{Z}_A(K) \in A$ be the element obtained from multiplying from right to left the beads along the diagram following the orientation of the knot (adding summations  accounting for the different copies of $R$). For instance, in the case of the figure-of-eight knot $4_1$ we have

\vspace{0.3cm}\noindent
 \begin{minipage}{.45\textwidth}
 \begin{equation*} 
\labellist \small  \hair 2pt
\pinlabel{$ \color{rosa} \bullet$} at 222 162
\pinlabel{$ \color{rosa} \alpha_i$} at 140 162
\pinlabel{$ \color{rosa} \bullet$} at 308 162
\pinlabel{$ \color{rosa} \beta_i$} at 390 162
\pinlabel{$ \color{blue} \bullet$} at 347 305
\pinlabel{$ \color{blue} \bar{\alpha}_j$} at 520 305
\pinlabel{$ \color{blue} \bullet$} at 437 305
\pinlabel{$ \color{blue} \bar{\beta}_j$} at 275 305
\pinlabel{$ \color{OliveGreen} \bullet$} at 226 461
\pinlabel{$ \color{OliveGreen} \alpha_\ell$} at 140 461
\pinlabel{$ \color{OliveGreen} \bullet$} at 304 461
\pinlabel{$ \color{OliveGreen} \beta_\ell$} at 370 461
\pinlabel{$ \color{Goldenrod} \bullet$} at 343 605
\pinlabel{$ \color{Goldenrod} \bar{\alpha}_r$} at 510  605
\pinlabel{$ \color{Goldenrod} \bullet$} at 435 605
\pinlabel{$ \color{Goldenrod} \bar{\beta}_r$} at 270 605
\pinlabel{$ \color{orange} \bullet$} at 2 362
\pinlabel{$ \color{orange} \kappa^{-1}$} at -90 362
\pinlabel{$ \color{orange} \bullet$} at 647 513
\pinlabel{$ \color{orange} \kappa$} at 699 513
\endlabellist
\includegraphics[width=0.45\textwidth]{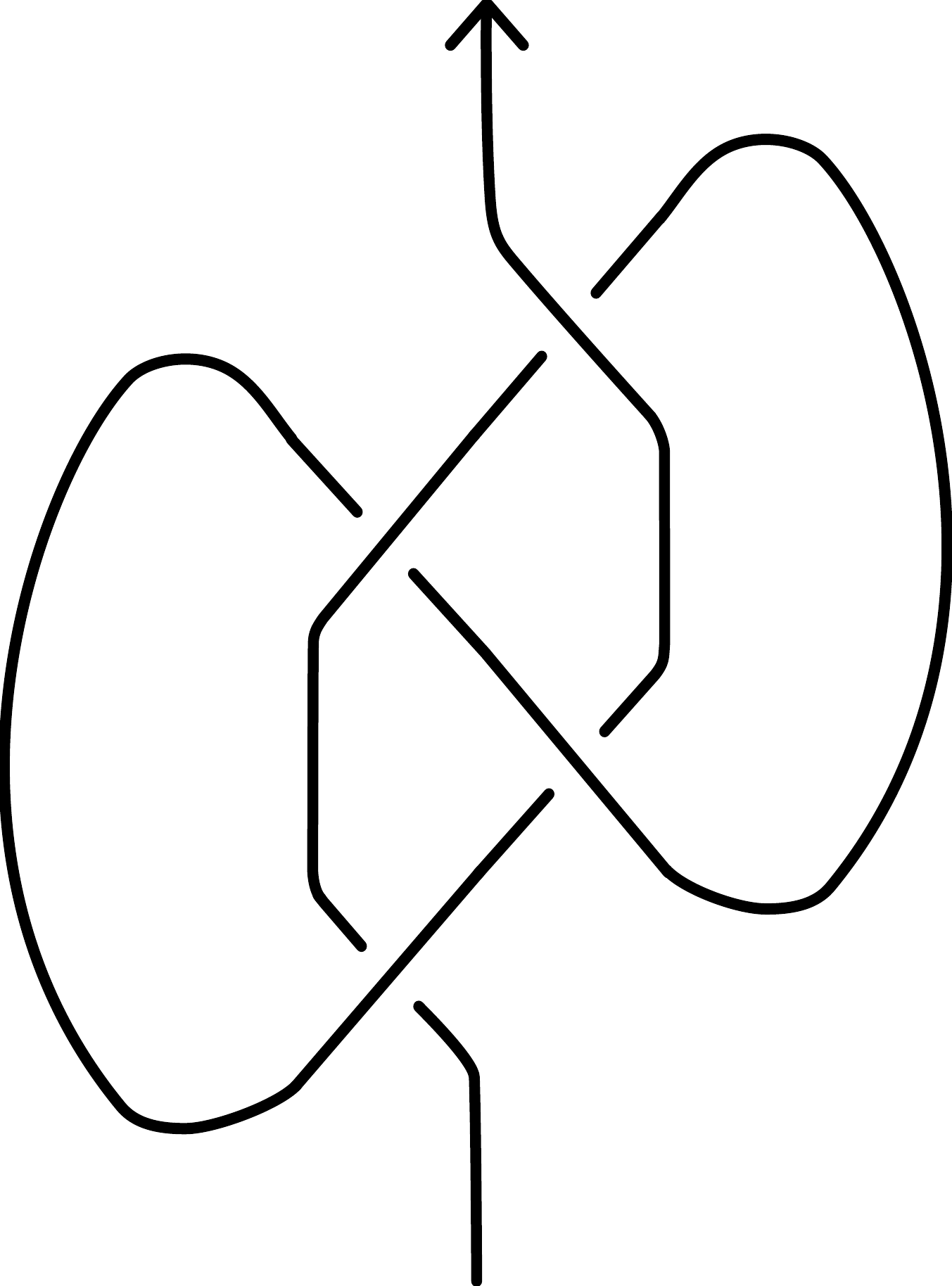} 
\end{equation*}
 \end{minipage}
  \begin{minipage}{.45\textwidth}
 $$\mathfrak{Z}_A(4_1) = \sum_{\color{rosa}{i} \color{black}{,} \color{blue}{j} \color{black}{,} \color{OliveGreen}{\ell} \color{black}{,} \color{Goldenrod}{r}} 
  \color{Goldenrod}{\bar{\alpha}_r} \  
  \color{blue}{\bar{\beta}_j}   \   
           \color{rosa}{\alpha_i} \
        \color{orange}{\kappa^{-1}} \
   \color{OliveGreen}{\beta_\ell}  \
     \color{blue}{\bar{\alpha}_j} \ 
             \color{orange}{\kappa} \
               \color{Goldenrod}{\bar{\beta}_r} \ 
               \color{OliveGreen}{\alpha_\ell}\ 
      \color{rosa}{\beta_i} \
            \color{black}{.}
$$
 \end{minipage}
 \vspace{0.3cm}

\noindent The reason why XC-algebras are the minimal algebraic structure needed to produce a knot invariant using Lawrence's construction is simply that the axioms  (XC0) -- (XC3) together with the invertibility of $R$ and $\kappa$ are precisely the algebraic counterparts via $\mathfrak{Z}_A$ of the rotational Reidemeister moves displayed in  \cite[\S 4.3]{BH_reidemeister}. Note that $\mathfrak{Z}_A (\begin{array}{c}
\centering
\includegraphics[scale=0.035]{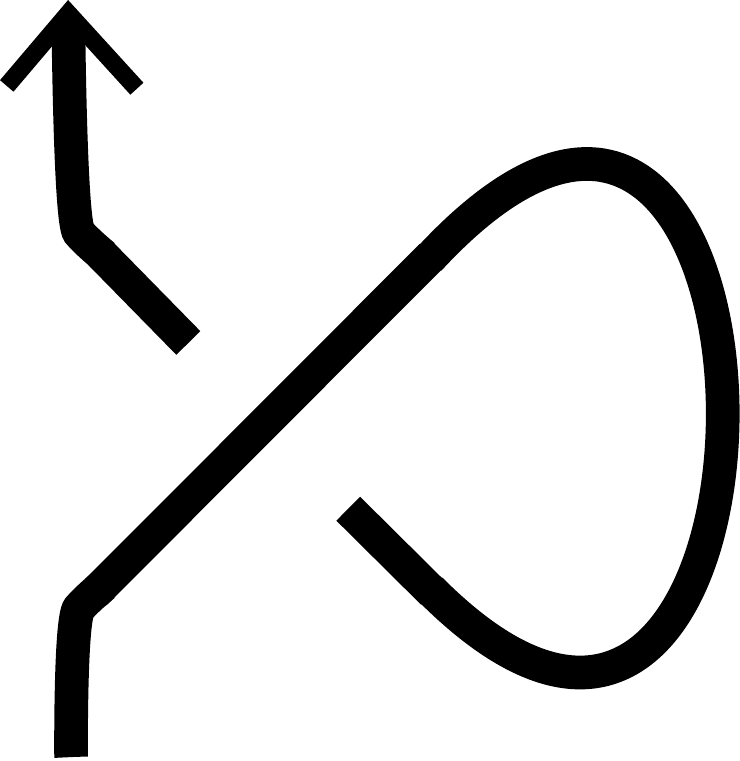} \end{array}) = \nu$.

The goal of this paper is to study XC-structures on Sweedler's four-dimensional complex  algebra $$SW:= \langle s,w | s^2=1, \ w^2=0, \ sw=-ws  \rangle,$$ and the universal invariant derived from them.  XC-structures on $SW$ are given by the solutions of a system of non-linear equations in 40 unknowns. 
We exhibit several multi-parameter families of  XC-structures on $SW$ that do not have a ribbon Hopf-algebraic origin, that is, that do not come from a ribbon Hopf algebra structure on $SW$ or from an endomorphism algebra of another ribbon Hopf algebra, see \cref{ex:1} -- \ref{ex:4}. 

The universal $R$-matrix of the standard ribbon Hopf algebra structure on $SW$ is triangular and therefore  produces a trivial knot invariant, see \cref{lem:triangular} for a general proof of this fact. However, most  XC-structures on $SW$ are not triangular. The main result of this note is that, surprisingly,  \textit{any} XC-structure on the Sweedler algebra produces a trivial invariant:

\begin{theorem}[\cref{thm:main}]
Any XC-algebra structure on the Sweedler algebra produces a framed knot invariant that only depends on  the framing:
$$ \mathfrak{Z}_{SW}(K)= \nu^{\mathrm{fr}(K)} . $$ In particular, this invariant is trivial for any 0-framed knot.
\end{theorem}

That is, despite $SW$ being non-commutative, the universal invariant subject to any XC-structure on $SW$ behaves as one in a commutative algebra, compare with  \cref{prop:comm_XC}. As immediate consequence, any ribbon Hopf algebra structure on $SW$ and any Reshetikhin-Turaev invariant obtained from a representation of one of these  will also return a trivial knot invariant.

\subsection*{Acknowledgments} The author would like to thank Dror Bar-Natan and Roland van der Veen for creating and making available the Mathematica implementation to compute universal knot invariants. The author was supported by the ARN project CPJ number ANR-22-CPJ1-0001-0  at the Institut de Mathématiques de Bourgogne (IMB). The IMB receives support from the EIPHI Graduate School (contract ANR-17-EURE-0002).

%%%%%%%%%%%%%%%%%%%%%%%%%%%%%%%%%%%%

\section{Commutative XC-algebras}

We would like to start by giving a classification result about XC-structures on commutative algebras and describing the universal invariant for each of them. This will be instructing since, as we explained in the Introduction, the main purpose of this note is to show that the universal invariant subject to any XC-structure on the Sweedler algebra behaves as one in a commutative algebra.

\begin{proposition}\label{prop:comm_XC}
Let $A$ be a commutative algebra, let $R \in A\otimes A$ be an arbitrary invertible element  and let $\kappa \in A$ be a square root of unity. Then the triple $(A,R, \kappa)$ is an XC-algebra. Furthermore, any XC-structure over $A$ is of this form.
\end{proposition}
\begin{proof}
First we note that $\kappa$ is invertible with inverse $\kappa^{-1}= \kappa$. The axioms (XC0) -- (XC2d) are automatic, for instance for the latter we have 
$$\sum_{i,j}   \bar{\alpha}_i \kappa  \alpha_j \otimes \beta_j \bar{\beta}_i = (\kappa \otimes 1)\left( \sum_{i,j}   \bar{\alpha}_i  \alpha_j \otimes  \bar{\beta}_i  \beta_j\right) = (\kappa \otimes 1)  R^{-1} R= \kappa \otimes 1.$$ Lastly for the Yang-Baxter equation (XC3)
\begin{align*}
R_{12}R_{13}R_{23} &= \sum_{i,j,k} \alpha_i \alpha_j \otimes \beta_i \alpha_k \otimes \beta_j \beta_k  = \sum_{i,j,k}  \alpha_j \alpha_i \otimes  \alpha_k \beta_i\otimes \beta_k \beta_j \\ &= \sum_{i,j,k}  \alpha_j \alpha_k \otimes  \alpha_i \beta_k\otimes \beta_i \beta_j =  R_{23}R_{13}R_{12}.
\end{align*}
For the last part of the statement, we first consider the element $\theta:= \sum_i \alpha_i \beta_i$. We claim that $\theta$ is invertible with inverse $u:= \sum_i \bar{\alpha}_i \bar{\beta}_i$, indeed this is an immediate consequence of the equations $R R^{-1}=1 \otimes 1 = R^{-1}R$. The axiom (XC1f) then implies that $$\kappa \theta = \sum_i \beta_i \kappa \alpha_i = \sum_i \alpha_i \kappa^{-1} \beta_i  = \kappa^{-1} \theta , $$ that is, $\kappa$ is a square root of unity.
\end{proof}

\begin{remark}
If $A$ is a commutative ribbon Hopf algebra, then the element $u$ in the proof is precisely the Drinfeld element $u:= \sum_i S(\beta_i) \alpha_i$  of the structure. Indeed, since $A$ is commutative, the antipode $S$ is an involution, so 
$$\sum_i S(\beta_i) \alpha_i = \sum_i S^{-1}(\beta_i) \alpha_i = \sum_i \bar{\beta}_i \bar{\alpha}_i = \sum_i  \bar{\alpha}_i \bar{\beta}_i$$ using the well-known equality $(\id \otimes S^{-1})(R)=R^{-1}$.
\end{remark}

For commutative XC-algebras, the universal invariant turns out to be very weak. We believe this is known to experts in the ribbon Hopf algebra setting but we include a proof due to lack of reference.

\begin{theorem}\label{thm:invariant_comm_XC}
Let $(A,R,\kappa)$ be a commutative XC-algebra. Then for any (long) knot $K$, its universal invariant $\mathfrak{Z}_A(K)$ is fully determined by its framing $\mathrm{fr}(K)$; more precisely, we have
$$\mathfrak{Z}_A(K) = \nu^{\mathrm{fr}(K)}.$$
In particular, the universal invariant is trivial for any 0-framed knot.
\end{theorem}
\begin{proof}
Let $D$ be a rotational diagram of $K$. Recall from \cite{becerra_refined} that the rotation number $rot (D)$ of $D$ is the difference between the number of positive full rotations $C$ and negative full rotations $C^-$.  If $\theta= \sum_i \alpha_i \beta_i$ and $u= \sum_i \bar{\alpha}_i \bar{\beta}_i$ as above, and $n_\pm$ denotes the number of positive and negative crossings in the diagram,  we have the following equalities that we explain below:
\begin{align*}
\mathfrak{Z}_A(K)&= \kappa^{-rot (D)} \cdot \theta^{n_+} \cdot u^{n_-} \\
&= \kappa^{-rot (D)}\cdot \theta^{wr(D)}\\
&= \kappa^{2r+wr(D)}\cdot \theta^{wr(D)}\\
&= \kappa^{wr(D)}\cdot \theta^{wr(D)}\\
&= \nu^{\textrm{fr}(K)}.
\end{align*}
The first equality is a consequence of the commutativity of $A$, the second one is simply the definition of the writhe $wr(D)=n_+ - n_-$ and the fact that $u= \theta^{-1}$ as explained in the proof of \cref{prop:comm_XC}, the third is \cite[Corollary 3.7]{becerra_refined} stating that the sum $rot(D)+w(D)$ is always even, the forth follows from the fact that $\kappa$ is a square root of unity, and the fifth  the fact that $\nu= \theta \kappa$ and the fact that the writhe of a diagram is precisely the framing of the knot it represents.
\end{proof}

\section{The Sweedler algebra}

Let $SW$ be the \textit{Sweedler algebra}, that is, the 4-dimensional complex algebra with two generators $s,w$ with relations $$ s^2=1 \qquad , \qquad w^2=0 \qquad , \qquad sw=-ws. $$
This is the smallest non-semisimple $\mathbb{C}$-algebra that admits a Hopf algebra structure (in fact such a Hopf algebra structure is unique up to isomorphism).

For convenience, let us write $p := (1-s)/2$.
It is well-known that $SW$ is a non-commutative, non-cocommutative ribbon Hopf algebra with universal $R$-matrix $$R_\lambda=1 \otimes 1  - 2 p \otimes p + \lambda ( w \otimes w + 2 wp \otimes wp - 2 w \otimes wp ) \qquad , \qquad \lambda \in \mathbb{C}$$ and balancing element $\kappa := s$ (this is a one-parameter family of ribbon structures), see e.g. \cite[Example 2.1.7]{majid_foundations} or \cite[Example 2.4.5]{becerra_thesis}. This universal $R$-matrix is triangular in the sense that $R^{-1}=R_{21}$. It is a folklore result that the universal invariant built out of this structure as well as the Reshetikhin-Turaev invariant for any finite-dimensional representation of $SW$ will produce weak invariants that only depend on the framing. This follows at once in the setting of XC-algebras from the results of \cite{becerra_refined} and the lemma below. For that, we will say that an XC-algebra $(A,R, \kappa)$ is  \textit{triangular} if $R^{-1}=R_{21}$.

\begin{lemma}\label{lem:triangular}
Let $(A,R, \kappa)$  be a triangular XC-algebra. Then for any (long) knot $K$, its universal invariant $\mathfrak{Z}_A(K)$ is fully determined by $r:= \mathrm{fr}(K) \pmod 2$: we have 
$$\mathfrak{Z}_A(K) = \nu^{r}.$$
In particular, the universal invariant is trivial for any 0-framed knot.
\end{lemma}
\begin{proof}
The condition $\sum_{i} \bar{\alpha}_i \otimes \bar{\beta}_i = \sum_i \beta \otimes \alpha$ means that $\mathfrak{Z}_A(D) = \mathfrak{Z}_A(D')$ where $D'$ is a rotational knot diagram obtained from another such $D$ by arbitrary replacing positive crossings for negative crossings and vice versa. Since any knot diagram can be turned into a diagram of the unknot (with a certain framing) by changing crossings signs, it means that $\mathfrak{Z}_A(K) $ must be a power of $\nu$, namely the value of a concatenation of positive and negative curls.  Furthermore, triangularity implies that $ \nu = \mathfrak{Z}_A (\begin{array}{c}
\centering
\includegraphics[scale=0.032]{twist_pos} \end{array}) = \mathfrak{Z}_A(\begin{array}{c}
\centering
\includegraphics[scale=0.032]{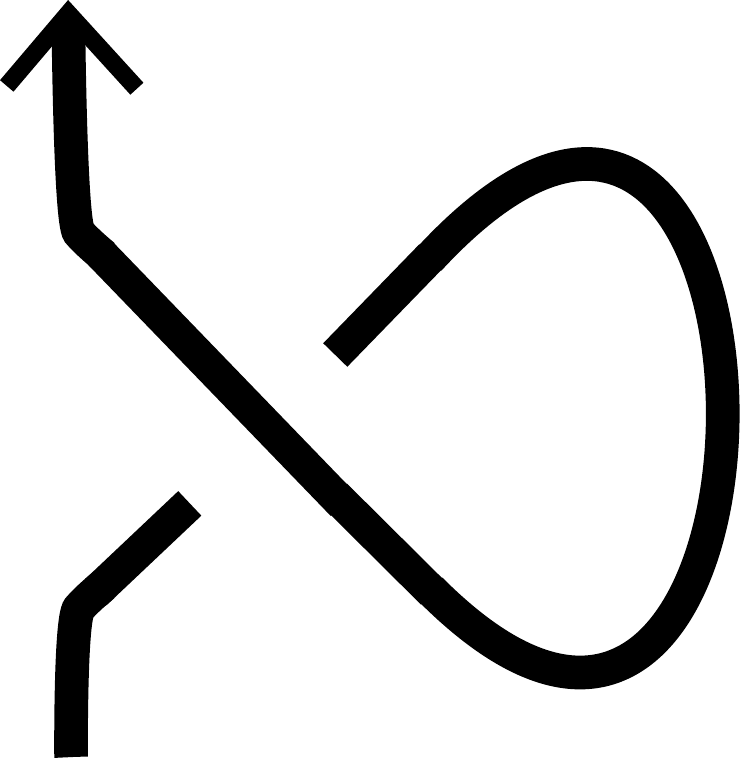} \end{array}) = \nu^{-1}$, that is, that $\nu$ is a square root of unity and therefore $\mathfrak{Z}_A(K) = \nu^r$ with $r=0,1$. That $r= \mathrm{fr}(K) \pmod 2$ follows from the fact that changing the sign of a crossing in a given diagram preserves its writhe modulo 2.
\end{proof}

However, it remains the question of whether the Sweedler algebra admits non-triangular XC-algebra structures that we can use to produce potentially non-trivial  knot invariants.  We now exhibit several multi-parameter families of XC-structures on $SW$:

%\begin{example}
%Let $\lambda, \mu, \gamma \in \mathbb{C}$ with $\lambda \neq \pm \mu$. Then 
%\begin{align*}
%R_{\lambda, \mu, \gamma}  &:= \mu (s\otimes s - 1 \otimes 1) + \lambda s \otimes 1  + \gamma (w \otimes w -  sw \otimes sw) \\ &\phantom{=} - \sqrt{-\lambda^2 + 2 \mu^2} \  1 \otimes s + sw \otimes w - w \otimes sw 
%\end{align*}
%and $\kappa:= s$ define three-parameter family of XC-algebra structures on $SW$.
%\end{example}

\begin{example}\label{ex:1}
Let $\lambda \in \mathbb{C}$ with $\lambda \neq 0, \pm 2$. Then \begin{align*}
R_{\lambda}  := \lambda \cdot 1 \otimes 1 + (1+s+w+sw)\otimes (s+w+sw)
\end{align*}
and $\kappa:= -s-w-sw$ define a one-parameter family of XC-algebra structures on $SW$.
The inverse $R$-matrix is given by 
\begin{align*}
R_{\lambda}^{-1}  &= \frac{1}{4-\lambda^2} (1+s+w+sw)\otimes (s+w+sw) \\ &\phantom{=} + \frac{2}{\lambda (\lambda^2-4)}(s+w+sw)\otimes 1 + \frac{\lambda^2 -2}{\lambda (\lambda^2-4)} 1\otimes 1
\end{align*}
and $\kappa^{-1}=\kappa$. Note that this XC-algebra is not triangular, in particular $R_{\lambda }^{-1}$ is not a multiple of $(R_{\lambda})_{21}$.

We also remark that in this example $$ \mathfrak{Z}_{SW} (\begin{array}{c}
\centering
\includegraphics[scale=0.032]{twist_pos} \end{array}) = \nu = -1-(1+\lambda)(s+w+sw)$$ is \emph{not} central for $\lambda \neq -1$, which implies that this XC-algebra structure does not arise from a ribbon Hopf algebra structure on $SW$ by \cite[Proposition 8.2]{habiro}. On the other hand, $SW$ is not isomorphic, as an algebra, to a endomorphism algebra $\mathrm{End}_{\mathbb{C}}(V)$ for some 2-dimensional vector space (the latter is a simple algebra, $SW$ is not even semisimple), so that this XC-algebra structure does not arise either as the one induced by a 2-dimensional representation of a ribbon Hopf algebra.
\end{example}

\begin{example}\label{ex:2}
Let $\lambda, \mu, \gamma \in \mathbb{C}$ with $\lambda \neq 0$. Then 
\begin{align*}
R_{\lambda, \mu, \gamma}  &:= \lambda (s\otimes 1 - \mathbf{i} \cdot  1 \otimes s) - \lambda \gamma (w \otimes 1 - \mathbf{i} \cdot  1 \otimes w)  + \mu(w \otimes w -  sw \otimes sw) \\ &\phantom{=}  + sw \otimes w - w \otimes sw 
\end{align*}
and $\kappa:= -s+\gamma w$ define a three-parameter family of XC-algebra structures on $SW$. The inverse $R$-matrix is given by 
\begin{align*}
R_{\lambda, \mu, \gamma}^{-1}  &= \frac{1}{2 \lambda^2} (\lambda (s\otimes 1 + \mathbf{i} \cdot  1 \otimes s) - \lambda \gamma (w \otimes 1 + \mathbf{i} \cdot  1 \otimes w)  + \mu \mathbf{i}  (-w \otimes w +  sw \otimes sw) \\ &\phantom{=}  + \mathbf{i} (- sw \otimes w + w \otimes sw) )
\end{align*}
and $\kappa^{-1}=\kappa$. As in the previous example, this XC-algebra is not triangular; in fact $R_{\lambda, \mu, \gamma}^{-1}$ is not a multiple of $(R_{\lambda, \mu, \gamma})_{21}$.
\end{example}

\begin{example}\label{ex:3}
Let $\lambda_1 , \ldots , \lambda_6 \in \mathbb{C}$ with $\lambda_1 \neq 0$. Then \begin{align*}
R_{\lambda_1 , \ldots , \lambda_6}  &:= \lambda_1 \cdot 1 \otimes 1 + \lambda_2 \cdot 1 \otimes w + \lambda_3 \cdot w \otimes w + \lambda_4 \cdot 1 \otimes sw \\ &\phantom{=}  + \lambda_5 \cdot w \otimes sw + \lambda_6  \cdot sw \otimes sw +  sw \otimes w
\end{align*}
and $\kappa:= 1$ define a six-parameter family of XC-algebra structures on $SW$. Again the XC-algebra is not triangular and $R_{\lambda_1 , \ldots , \lambda_6}^{-1}$ is not a multiple of $(R_{\lambda_1 , \ldots , \lambda_6})_{21}$: we have
\begin{align*}
R_{\lambda_1 , \ldots , \lambda_6}^{-1}  &= \frac{1}{\lambda_1} 1 \otimes 1 - \frac{1}{\lambda_1^2} ( \lambda_2 \cdot 1 \otimes w + \lambda_3 \cdot w \otimes w + \lambda_4 \cdot 1 \otimes sw \\ &\phantom{=}  + \lambda_5 \cdot w \otimes sw + \lambda_6  \cdot sw \otimes sw +  sw \otimes w ).
\end{align*}
Note that similarly to \cref{ex:1}, we have in this case $$ \nu = \lambda_1 + \lambda_2 w + \lambda_4 sw$$ which  is \emph{not} central unless $\lambda_2=\lambda_4=0$, which again means  that this XC-algebra structure does not have a ribbon Hopf-algebraic origin.
\end{example}

\begin{example}\label{ex:4}
Let $\lambda_1 , \ldots , \lambda_5 \in \mathbb{C}$ with $\lambda_1 \neq 0$. Then \begin{align*}
R_{\lambda_1 , \ldots , \lambda_5}  &:= \lambda_1 \cdot 1 \otimes 1 + \lambda_2 (sw \otimes 1 -1 \otimes sw + 1 \otimes w - w \otimes 1) \\ &\phantom{=}  + \lambda_3 ( w \otimes s - s \otimes w -sw \otimes s )+ \lambda_4 \cdot w \otimes w \\ &\phantom{=}  + \lambda_5 \cdot sw \otimes w + \frac{2 \lambda_2^2 - 2 \lambda_3^2 - \lambda_1 \lambda_5}{\lambda_1}  w \otimes sw + \frac{2 \lambda_3^2 - 2 \lambda_2^2 - \lambda_1 \lambda_4}{\lambda_1}  sw \otimes w
\end{align*}
and $\kappa:= 1+2\frac{\lambda_3}{\lambda_1}(w-sw)$ define a five-parameter family of XC-structures on $SW$. Once more $R_{\lambda_1 , \ldots , \lambda_5}^{-1}$ is not a multiple of $(R_{\lambda_1 , \ldots , \lambda_5})_{21}$. In this case we note that $\kappa^{-1} \neq \kappa$.
\end{example}

It is natural to study the framed knot invariants produced by these XC-structures on $SW$. They are all non-triangular and besides the Sweedler algebra is non-commutative, so a priori they could produce any kind of numeral invariant: there is no reason to think that they will produce trivial invariants, depending only on the framing.

Strikingly, one checks computationally that they all do. In fact, so will any other XC-structure on the Sweedler algebra. This is the content of our main result:

\begin{theorem}\label{thm:main}
Any XC-algebra structure on the Sweedler algebra produces a framed knot invariant that only depends on  the framing:
$$ \mathfrak{Z}_{SW}(K)= \nu^{\mathrm{fr}(K)} . $$ In particular, this invariant is trivial for any 0-framed knot.
\end{theorem}

From the results of \cite{becerra_refined} we automatically obtain

\begin{corollary}
Any ribbon Hopf algebra structure on $SW$ produces a trivial framed knot invariant (depending only on the framing). This is also the case for the Reshetikhin-Turaev invariant produced by any finite-dimensional representation of the Sweedler algebra equipped with any ribbon Hopf algebra structure.
\end{corollary}

The rest of the note is devoted to prove \cref{thm:main}.  We will break down the proof in several steps.

\begin{lemma}\label{lem:sqrt}
Let $e_1=1, e_2=s, e_3=w, e_4=sw$ and write $$ x= \sum_{i=1}^4 \lambda_i e_i \in SW \qquad , \qquad Q= \sum_{i,j=1}^4 \mu_{ij} e_i \otimes e_j \in SW \otimes SW . $$
%Let $x=\lambda_1 + \lambda_2 s + \lambda_3 w + \lambda_4 sw \in SW$  be an arbitrary element. Then
\begin{enumerate}
\item $x$ is invertible if and only if $\lambda_1 \neq \pm \lambda_2$, and in that case $$ x^{-1}= \frac{1}{\lambda_1^2 - \lambda_2^2}(\lambda_1 -  \lambda_2 s - \lambda_3 w - \lambda_4 sw ).$$
\item $Q$ is invertible if and only if $$ \mu_{11} + \mu_{12} \neq \pm (\mu_{21} +\mu_{22}  )  \qquad , \qquad  \mu_{11} - \mu_{12} \neq \pm (\mu_{21} -\mu_{22}  ) , $$ and $Q^{-1}$ can be written explicitly.
%$x$ is a square-zero element if and only if $\lambda_1= \lambda_2=0$.
\end{enumerate}
%In particular, we have that $x$ is a square root of unity if and only if either $x= \pm 1$ or $\lambda_1=0$ and $\lambda_2= \pm 1$.
\end{lemma}
\begin{proof}
Write $x'= \sum_{i=1}^4 \lambda'_i e_i$ and  $ Q'= \sum_{i,j=1}^4 \mu_{ij}' e_i \otimes e_j $. Viewing the equation $xx'=1$ (resp. $QQ'=1\otimes 1$) as a system of linear equations with coefficients $\lambda_i$ (resp. $\mu_{ij}$) viewed as formal variables  and unknowns  $\lambda'_i$ (resp. $\mu_{ij}'$), we obtain by direct computation that the determinant of the matrix of the system of linear equations equals $-\left(\lambda _1^2-\lambda _2^2\right){}^2$ (resp.
\begin{align*}
- &\left(\mu _{1,1}+\mu _{1,2}-\mu _{2,1}-\mu _{2,2}\right){}^4 \left(\mu _{1,1}-\mu _{1,2}+\mu _{2,1}-\mu _{2,2}\right){}^4 \\ &\times \left(\mu _{1,1}-\mu _{1,2}-\mu _{2,1}+\mu _{2,2}\right){}^4 \left(\mu _{1,1}+\mu _{1,2}+\mu _{2,1}+\mu _{2,2}\right){}^4  \quad ).
\end{align*}
The formulas for $x'=x^{-1}$ and $Q'=Q^{-1}$ are obtained from the standard formula of the inverse matrix using the adjugate matrix. See \cref{appendix} for an explicit formula of $Q^{-1}$.
\end{proof}

%\begin{lemma}
%Let $x \in SW$ be a square-zero element. Then $xyx=0$ for any $y \in SW$.
%\end{lemma}
%\begin{proof}
%It is well-known that the Jacobson radical $J$ of the Sweedler algebra is given by the linear span $J= \mathrm{span}_{\mathbb{C}}(w,sw)$. Then \cref{lem:sqrt}.(2) implies that $x \in J$. Since $J$ is a two-sided ideal (being the intersection of the annihilator ideals, which are two-sided), then $xy \in J$. We conclude that $xyx=0$ as $J^2=0$.
%\end{proof}

Let us write  $J= \mathrm{span}_{\mathbb{C}}(w,sw)$, and note that $J$ is in fact a two-sided ideal of $SW$ ($J$ is in fact the  \textit{Jacobson radical} of $SW$).

%Recall that the of a ring is the two-sided ideal given by the intersection of the annihilators of all simple left modules. It is well-known that the Jacobson radical $J$ of the Sweedler algebra is given by the linear span $J= \mathrm{span}_{\mathbb{C}}(w,sw)$.

\begin{lemma}\label{lem:xzy}
If $x,y \in J$, then $xzy=0$ for any $z \in SW$.
\end{lemma}
\begin{proof}
 Since $J$ is a two-sided ideal, then $xz \in J$, so $xzy=0$ as clearly $J^2=0$.
\end{proof}

Let us denote  $[u,v]:=uv-vu$ for the commutator of elements of $SW \otimes SW$.

\begin{proposition}\label{prop:commutators}
For any XC-structure $(R, \kappa)$ on $SW$, we have that the commutators
$$ [R,\kappa^{\pm 1} \otimes 1 ] \quad, \quad  [R, 1 \otimes \kappa^{\pm 1}] \quad, \quad [R^{-1},\kappa^{\pm 1} \otimes 1] \quad, \quad  [R^{-1}, 1 \otimes \kappa^{\pm 1}] $$ belong all to $J \otimes J$.
\end{proposition}
\begin{proof}
Let $e_1=1, e_2=s, e_3=w, e_4=sw$ and write $$ \kappa= \sum_{i=1}^4 \lambda_i e_i \qquad , \qquad R= \sum_{i,j=1}^4 \mu_{ij} e_i \otimes e_j \qquad , \qquad R^{-1}= \sum_{i,j=1}^4 \mu_{ij}' e_i \otimes e_j  . $$
Since $\kappa$ must be invertible, \cref{lem:sqrt} implies that $\lambda_1 \neq \pm \lambda_2$ and
\begin{equation}\label{eq:kappa-1}
\kappa^{-1}= \frac{1}{\lambda_1^2 - \lambda_2^2}(2\lambda_1 -   \kappa) .
\end{equation}

On the other hand, let us write
$$ (\kappa \otimes \kappa) \cdot R^{\pm 1} \cdot (\kappa^{-1} \otimes \kappa^{-1})-R^{\pm 1} = \sum_{i,j} p_{i,j}^\pm e_i \otimes e_j,  $$
where $p_{i,j}^\pm$ are viewed as rational functions in $\lambda_i, \mu_{ij}$ and $\mu'_{ij}$ where the only denominator appearing is $\lambda_1^2 - \lambda_2^2$ or its square.

The proof will consist in showing that the relations  $p_{ij}^\pm=0$, consequence of (XC0), imply that, for each of these commutators,  the coordinates corresponding to the basis elements outside $J \otimes J= \mathrm{span}_{\mathbb{C}}(e_3 \otimes e_3, e_3 \otimes e_4, e_4 \otimes e_3,e_4 \otimes e_4)$ vanish. Let us denote $\pi: SW \otimes SW \to  (SW \otimes SW)/(J \otimes J)$ for the natural projection to the quotient, that we identify with the linear subspace spanned by the rest of basis elements $e_i \otimes e_j$ not in $J \otimes J$.

By direct computation we find that 
\begin{align*}
\pi ([R,\kappa \otimes 1 ]) &= 2(  \lambda _4 \mu _{1,2}- \lambda _2 \mu _{1,4})w \otimes 1  + 2( \lambda _3 \mu _{1,2}- \lambda _2 \mu _{1,3})sw \otimes 1 \\
&+ 2( \lambda_4 \mu _{2,2}- \lambda _2 \mu _{2,4}) w \otimes s+ 2( \lambda _3 \mu _{2,2}- \lambda _2 \mu _{2,3}) sw \otimes s  .
\end{align*}
For the first of the coefficients, we first compute 
\begin{align*}
p_{3,1}^+ &= \frac{2}{\lambda_1^2 - \lambda_2^2} ( \lambda _2^2 \mu _{1,3}- \lambda _3 \lambda _2 \mu _{1,2}+ \lambda _1 \lambda _2 \mu _{1,4}- \lambda _1 \lambda _4 \mu _{1,2})\\
p_{4,1}^+ &=  \frac{2}{\lambda_1^2 - \lambda_2^2} ( \lambda _2^2 \mu _{1,4}- \lambda _4 \lambda _2 \mu _{1,2}+ \lambda _1 \lambda _2 \mu _{1,3}- \lambda _1 \lambda _3 \mu _{1,2})
\end{align*} 
and then we simply note that 
$$ 2(  \lambda _4 \mu _{1,2}- \lambda _2 \mu _{1,4})= \lambda_2 p_{4,1}^+ - \lambda_1 p_{3,1}^+  =0 $$ by (XC0). This combination was found by computing a Gröbner basis for the $p_{ij}$ using Mathematica, treating $\lambda_1, \lambda_2$ as constants as they appear in denominators.  The vanishing of the rest of the coefficients as well as  the corresponding coefficients of $\pi ([R,1 \otimes \kappa   ]) $ is computed similarly, see \cref{appendix}.

For $[R,\kappa^{- 1} \otimes 1 ] $ and $ [R, 1 \otimes \kappa^{- 1}]$, we simply observe that as a consequence of \eqref{eq:kappa-1} we have $$ [R,\kappa^{- 1} \otimes 1 ] = -\frac{1}{\lambda_1^2 - \lambda_2^2} [R,\kappa \otimes 1 ]  \qquad , \qquad  [R,1 \otimes \kappa^{- 1}] = -\frac{1}{\lambda_1^2 - \lambda_2^2} [R,1 \otimes \kappa ] , $$
and applying $\pi$ we conclude. The statement for $ [R^{-1},\kappa^{\pm 1} \otimes 1] $ and $ [R^{-1}, 1 \otimes \kappa^{\pm 1}]$  follows verbatim using the version of (XC0) with $R^{-1}$, now with the coefficients $\mu_{ij}'$.
\end{proof}
%$(J \otimes J)^\bot$ the linear subspace spanned by the rest of basis elements $e_i \otimes e_j$ not in $J \otimes J$, so that $(J \otimes J) \oplus (J \otimes J)^\bot = SW \otimes SW$, and let us write 

For $k >0$, let $\mathfrak{S}_k$ be the symmetric group of order $k$. If $\sigma \in \mathfrak{S}_k$, let us write $$ \mu_\sigma : SW^{\otimes k} \to SW \qquad , \qquad x_1 \otimes \cdots \otimes x_k \mapsto  x_{\sigma^{-1}(1)} \otimes \cdots  \otimes x_{\sigma^{-1}(k)}.$$

\begin{proposition}\label{prop:permuting_R}
Let $n>0$ and $\sigma, \tau \in \mathfrak{S}_{2n}$. For any XC-structure $(R, \kappa)$ on $SW$, we have that $$  \mu_\sigma (R^{\otimes r} \otimes (R^{-1})^{\otimes n-r} ) = \mu_\tau (R^{\otimes r} \otimes (R^{-1})^{\otimes n-r} )$$ whenever $\sigma (2i)- \sigma (2i-1)$ and $\tau(2i)- \tau (2i-1)$ have the same sign for all $i=1, \ldots, n$.
\end{proposition}

In more simple terms, this means that in an expression of the form  $\mu_\sigma (R^{\otimes r} \otimes (R^{-1})^{\otimes n-r} )$, one can permute the order of any two consecutive Greek letters without changing the value of the expression so long as they have different indices. As a concrete example, this means that given an expression like $$ \sum_{i,j,k} \alpha_j \beta_k \bar{\alpha}_i \alpha_k \beta_j  \bar{\beta}_i  ,$$ the $\alpha_k$ can be placed anywhere to the right of the $\beta_k$, or the  $\beta_k$ anywhere to the left of $\alpha_k$, without altering the result of the sum (and similarly with the rest of letters).

\begin{proof}
First of all, we claim that  it suffices to check that  the differences
\begin{align*}
\sum_{i,j} (\tilde{\alpha}_i \tilde{\alpha}_j - \tilde{\alpha}_j \tilde{\alpha}_i) \otimes \tilde{\beta}_i \otimes \tilde{\beta}_j \quad &, \quad \sum_{i,j} (\tilde{\alpha}_i \tilde{\beta}_j - \tilde{\beta}_j \tilde{\alpha}_i) \otimes \tilde{\beta}_i \otimes \tilde{\alpha}_j, \\
\sum_{i,j} (\tilde{\beta}_i \tilde{\alpha}_j - \tilde{\alpha}_j \tilde{\beta}_i) \otimes \tilde{\alpha}_i \otimes \tilde{\beta}_j \quad &, \quad \sum_{i,j} (\tilde{\beta}_i \tilde{\beta}_j - \tilde{\beta}_j \tilde{\beta}_i) \otimes \tilde{\alpha}_i \otimes \tilde{\alpha}_j, 
\end{align*} 
where the tilde stands for a bar or nothing, belong all to $$U:=(J \otimes J \otimes A) \oplus (J \otimes A \otimes J) \oplus (A  \otimes J \otimes J ).$$ Indeed given an element of the form $\mu_\sigma (R^{\otimes r} \otimes (R^{-1})^{\otimes n-r} )$, the difference between such an element and the element obtained by permuting two consecutive Greek letters with different indices can be written as a certain product of one of the four elements above and other copies of $R^{\pm 1}$. Since $J$ is a two-sided ideal, this implies that the above-mentioned difference lies in $J^2=0$ and therefore the two elements are equal.

Let us now prove the claim, using the notation of \cref{prop:commutators}. There are twelve cases to analyse; we explain here the case $\Gamma:=\sum_{i,j} (\alpha_i \alpha_j - \alpha_j \alpha_i) \otimes \beta_i \otimes \beta_j $ in detail and refer the reader to \cref{appendix} for the rest of cases. Let us write $\pi: SW^{\otimes 3} \to SW^{\otimes 3}/U$ for the natural projection, and let us identify the quotient with the linear subspace of $SW^{\otimes 3}$ spanned by the basis elements $e_i \otimes e_j \otimes e_k$ not in $U$. By direct computation we obtain
\begin{align*}
\pi(\Gamma) &= (2 \mu _{1,4} \mu _{2,2}-2 \mu _{1,2} \mu _{2,4}) w \otimes w \otimes 1 + (2 \mu _{1,3} \mu _{2,2}-2 \mu _{1,2} \mu _{2,3})sw\otimes s \otimes 1 \\
&+ (2 \mu _{1,2} \mu _{2,4}-2 \mu _{1,4} \mu _{2,2})w \otimes 1 \otimes s  + (2 \mu _{1,2} \mu _{2,3}-2 \mu _{1,3} \mu _{2,2})sw \otimes 1 \otimes s.
\end{align*}
To see that each of these coefficients vanishes, we argue as follows: if $\lambda_2 \neq 0$, we directly see that these coefficients can be written as a linear combination of the elements of the Gröbner basis computed in \cref{prop:commutators} from (XC0). If $\lambda_2=0$, then this Gröbner basis forces that either $\lambda_2=\lambda_3= \lambda_4=0$ or $\mu_{12}=\mu_{21}=\mu_{22}=0$. In each of these two cases, we obtain a finite number of (parametric) solutions for the equations determined by (XC0) -- (XC3).  For each of these solutions, the coefficients of $\pi(\Gamma)$ are computed to be zero.
\end{proof}

We have now all the technical ingredients to prove the main theorem of this note. Roughly, the proof will consist in noting that the two propositions above allow us to  use an argument very similar to that of \cref{thm:invariant_comm_XC} in the commutative case.

\begin{proof}[Proof (of \cref{thm:main})]
Let $D$ be a rotational diagram of $K$, say with $n_+$ positive crossings, $n_-$ negative crossings, $m_+$ positive full rotations and $m_-$ negative full rotations. By definition,
\begin{equation}\label{eq:expr}
\mathfrak{Z}_{SW}(K) = \mu_\sigma (R^{\otimes n_+} \otimes (R^{-1})^{\otimes n_-} \otimes (\kappa^{-1})^{\otimes  m_+} \otimes  \kappa ^{\otimes  m_-} ) 
\end{equation}
for some $\sigma \in \mathfrak{S}_{2(n_+ + n_-)+ m_+ + m_-}$ determined by the diagram $D$.

The first step of the proof is to show that
\begin{equation}\label{eq:expr_2}
\mathfrak{Z}_{SW}(K) = \kappa^{- rot(D)} \cdot \mu_{\sigma'} (R^{\otimes n_+} \otimes (R^{-1})^{\otimes n_-} ) 
\end{equation}
where $\sigma' \in \mathfrak{S}_{2(n_+ + n_-)}$ is the permutation obtained from $\sigma$ by removing the last $m_+ + m_-$ indices and shifting the rest accordingly in an orderly fashion. Indeed we  claim that we can move each of the copies of $\kappa^{\pm 1}$ appearing in the summation \eqref{eq:expr}, once at a time, to the front without altering the value of $\mathfrak{Z}_{SW}(K)$. Starting by the leftmost instance of $\kappa^{\pm 1}$  in  \eqref{eq:expr}, let us suppose without loss of generality that the letter appearing before $\kappa^{\pm 1}$ is $\alpha_i$ and that $\beta_i$ appears to the right, $$\mathfrak{Z}_{SW}(K)= \sum \cdots \alpha_i \kappa^{\pm 1} \cdots \beta_i \cdots ,$$ where the dots refer to other copies of $R^{\pm 1}$ and $\kappa^{\pm 1}$. By \cref{prop:commutators}, this equals
$$\mathfrak{Z}_{SW}(K)= \sum \cdots \kappa^{\pm 1} \alpha_i  \cdots \beta_i \cdots  + \sum \cdots x_i  \cdots y_i \cdots $$
for some element $\sum_i x_i \otimes y_i \in J \otimes J$. By \cref{lem:xzy}, the second summand vanishes. We repeat the process until all copies of $\kappa^{\pm 1}$ are carried outside the summand. We conclude as $rot(D)=m_+ - m_-$.

Let us now show that $\mathfrak{Z}_{SW}(K)= \nu^{\mathrm{fr}(K)}$. Write $$ \theta := \sum_i \alpha_i \beta_i \qquad , \qquad \xi:= \sum_i \beta_i \alpha_i.  $$ As a consequence of (XC1f) and \cref{prop:commutators} we have
\begin{equation}\label{eq:theta=kappa2.xi}
\theta = \kappa^2 \xi
\end{equation}
(so $\theta$ and $\xi$ coincide if and only if $\kappa$ is a square root of unity). It immediately follows from the rotational Reidemeister move $\Omega 1 \text{f}a$ of \cite{BH_reidemeister}, \cref{prop:commutators} and \eqref{eq:theta=kappa2.xi} that
$$ \theta^{-1} = \sum_i \bar{\alpha}_i \bar{\beta}_i \qquad , \qquad \xi^{-1}= \sum_i \bar{\beta}_i \bar{\alpha}_i.  $$
Now, according to \cref{prop:permuting_R}, the expression $\mu_{\sigma'} (R^{\otimes n_+} \otimes (R^{-1})^{\otimes n_-} ) $ appearing in \eqref{eq:expr_2} equals another one where same indices appear consecutively, starting from the index of the leftmost Greek letter in the expression and bringing  the Greek letter with same index next to it, and proceeding with every index once at a time. Furthermore, \cref{prop:permuting_R} also implies that $\theta$ and $\xi$ commute.

%For instance for the figure-of-eight diagram from the introduction we have
%\begin{align*}
%\mathfrak{Z}_A(4_1) &= \sum_{\color{rosa}{i} \color{black}{,} \color{blue}{j} \color{black}{,} \color{OliveGreen}{\ell} \color{black}{,} \color{Goldenrod}{r}} 
%  \color{Goldenrod}{\bar{\alpha}_r} \  
%  \color{blue}{\bar{\beta}_j}   \   
%           \color{rosa}{\alpha_i} \
%        \color{orange}{\kappa^{-1}} \
%   \color{OliveGreen}{\beta_\ell}  \
%     \color{blue}{\bar{\alpha}_j} \ 
%             \color{orange}{\kappa} \
%               \color{Goldenrod}{\bar{\beta}_r} \ 
%               \color{OliveGreen}{\alpha_\ell}\ 
%      \color{rosa}{\beta_i} \\
%      &=\left( \sum_{ \color{Goldenrod}{r}} 
%  \color{Goldenrod}{\bar{\alpha}_r}
%               \color{Goldenrod}{\bar{\beta}_r}   \right) \left(   \sum_{\color{blue}{j} }   
%  \color{blue}{\bar{\beta}_j}      
%     \color{blue}{\bar{\alpha}_j} \right) \left(    
%     \sum_{\color{rosa}{i}} 
%           \color{rosa}{\alpha_i}  
%      \color{rosa}{\beta_i}
%     \right) \left(
%      \sum_{ \color{OliveGreen}{\ell} } 
%   \color{OliveGreen}{\beta_\ell}  \
%               \color{OliveGreen}{\alpha_\ell}
%     \right)
%\end{align*}

All in  all, we have the following chain of equalities, that we explain below:
\begin{align*}
\mathfrak{Z}_{SW}(K)&= \kappa^{- rot(D)} \cdot \mu_{\sigma'} (R^{\otimes n_+} \otimes (R^{-1})^{\otimes n_-} ) \\
&= \textstyle \kappa^{- rot(D)} \left( \sum_i \alpha_i \beta_i  \right)^{n_{+,1}} \left( \sum_i \beta_i  \alpha_i  \right)^{n_{+,2}} \left( \sum_i \bar{\alpha}_i \bar{\beta}_i  \right)^{n_{-,1}} \left( \sum_i  \bar{\beta}_i \bar{\alpha}_i \right)^{n_{-,2}} \\
&= \kappa^{- rot(D)} \theta^{n_{+,1}-n_{-,1}} \xi^{n_{+,2}-n_{-,2}}\\
&= \kappa^{- rot(D) + 2(n_{+,1}-n_{-,1})} \xi^{wr(D)}\\
&= \kappa^{wr(D)} \xi^{wr(D)}\\
&=\nu^{\mathrm{fr}(K)}.
\end{align*}
Here we have written $n_{\pm,1}$ (resp. $n_{\pm,2}$) for the number of positive/negative crossings where knot traverses the understrand (resp. the overstrand) in the first place. Clearly, $n_{\pm}=n_{\pm,1} + n_{\pm,2}$. The first equality in the above chain is noting but \eqref{eq:expr_2}, the second one follows from the discussion above using \ref{prop:permuting_R} (we remind the reader that we multiply the beads placed on the knot diagram from right to left). The third equality is a consequence of the commutativity of $\theta$ and $\xi$ noted above and the formulas for their inverses, the fourth one is simply \eqref{eq:theta=kappa2.xi}, and the fifth one follows directly from the equality 
$$ rot(D)+ wr(D)= 2(n_{+,1}-n_{-,1})  $$
for a rotational knot diagram $D$, which was shown in \cite[Corollary 3.7]{becerra_refined}. The last equality is a consequence of the equality $\nu= \kappa \xi$, which follows from \cref{prop:commutators}.
\end{proof}

\appendix

\section{Computations}\label{appendix}

A computer implementation in Mathematica \cite{Mathematica} with the calculations claimed in \cref{lem:sqrt},  \cref{prop:commutators} and  \cref{prop:permuting_R} can be found at

\begin{center}
\href{https://sites.google.com/view/becerra/XCsweedler}{\texttt{https://sites.google.com/view/becerra/XCsweedler}}.
\end{center}

%
%\begin{proposition}
%For any XC-structure $(R, \kappa)$ on $SW$, we have that $\kappa$ is a square root of unity.
%\end{proposition}
%\begin{proof}
%Let $e_1=1, e_2=s, e_3=w, e_4=sw$ and write $$ \kappa= \sum_{i=1}^4 \lambda_i e_i \qquad , \qquad R= \sum_{i,j=1}^4 \mu_{ij} e_i \otimes e_j . $$ From the axiom (XC1f), that is, $\mathfrak{Z}_{SW} (\begin{array}{c}
%\centering
%\includegraphics[scale=0.032]{twist_pos} \end{array}) = \mathfrak{Z}_{SW}(\begin{array}{c}
%\centering
% \scalebox{-1}[1]{\includegraphics[scale=0.032]{twist_neg}} \end{array}) $ we find comparing coefficients for $e_1$ and $e_2$ that
% \begin{align*}
% \frac{1}{\lambda_1^2 - \lambda_2^2} (\lambda_1(\mu_{12}+ \mu_{21}) - \lambda_2(\mu_{11}+ \mu_{22})  ) &= \lambda_1(\mu_{12}+ \mu_{21}) + \lambda_2(\mu_{11}+ \mu_{22})\\
%  \frac{1}{\lambda_1^2 - \lambda_2^2} (\lambda_1(\mu_{11}+ \mu_{22}) - \lambda_2(\mu_{12}+ \mu_{21})  ) &= \lambda_1(\mu_{11}+ \mu_{22}) + \lambda_2(\mu_{12}+ \mu_{21}).
%\end{align*}  
%Adding and subtracting the equations we find assuming that $\mu_{11}+ \mu_{22} \neq \pm(\mu_{12}+ \mu_{21})$ that $$ (\lambda_1 + \lambda_2)^2=1 \qquad , \qquad (\lambda_1 - \lambda_2)^2=1 ,$$ that is, that either $\lambda_1=0$ and $\lambda_2=\pm 1$ (in which case $\kappa$ is automatically a square root of unity by \cref{lem:sqrt}) or $\lambda_1= \pm 1$ and $\lambda_2=0$. It remains to check that in the latter case $\lambda_3= \lambda_4 =0$.
%\end{proof}

%\section{The Dilbert algebra}

%\nocite{*}
%\bibliographystyle{amsalpha}
%\bibliographystyle{alpha}
\bibliographystyle{halpha-abbrv}
\bibliography{bibliografia}

\end{document}